\documentclass[oneside,12pt, reqno]{amsart}

\usepackage[margin=.95in]{geometry}

\usepackage{etoolbox}
\usepackage{ifdraft}
\usepackage{iftex}

\ifluatex
\usepackage[utf8]{luainputenc}
\else
\usepackage[utf8]{inputenc}
\fi
\usepackage[T1]{fontenc}

\usepackage[english]{babel}
\usepackage{csquotes}

\usepackage{mathrsfs} %
\usepackage{dsfont}

\usepackage{amssymb}
\usepackage{mathtools}
\usepackage{nccmath}    %
\usepackage{tensor}

\usepackage{needspace}

\usepackage[backend=biber,style=numeric-comp,giveninits,url=false,sortcites=true,maxbibnames=99]{biblatex}
\usepackage[final,
            pdfborder={0 0 0}, colorlinks=true,
            linkcolor=BrickRed, citecolor=ForestGreen, urlcolor=RoyalBlue]{hyperref}

\usepackage[cmyk,dvipsnames]{xcolor}

\usepackage[textsize=scriptsize,colorinlistoftodos]{todonotes}
\ifdraft{%
\usepackage{lineno}
\usepackage{scrtime}
}{}

\usepackage{booktabs}
\usepackage[inline]{enumitem}
\usepackage{url}

\ifdraft{%
\linenumbers
\newcommand*\patchAmsMathEnvironmentForLineno[1]{%
  \expandafter\let\csname old#1\expandafter\endcsname\csname #1\endcsname
  \expandafter\let\csname oldend#1\expandafter\endcsname\csname end#1\endcsname
  \renewenvironment{#1}%
     {\linenomath\csname old#1\endcsname}%
     {\csname oldend#1\endcsname\endlinenomath}}%
\newcommand*\patchBothAmsMathEnvironmentsForLineno[1]{%
  \patchAmsMathEnvironmentForLineno{#1}%
  \patchAmsMathEnvironmentForLineno{#1*}}%
\AtBeginDocument{%
\patchBothAmsMathEnvironmentsForLineno{equation}%
\patchBothAmsMathEnvironmentsForLineno{align}%
\patchBothAmsMathEnvironmentsForLineno{flalign}%
\patchBothAmsMathEnvironmentsForLineno{alignat}%
\patchBothAmsMathEnvironmentsForLineno{gather}%
\patchBothAmsMathEnvironmentsForLineno{multline}%
}}

\AtEveryBibitem{\clearfield{doi}}
\AtEveryBibitem{\clearfield{isbn}}
\AtEveryBibitem{\clearfield{issn}}
\AtEveryBibitem{\clearfield{pages}}
\AtEveryBibitem{\clearlist{language}}

\renewcommand*{\bibfont}{\footnotesize}
\renewbibmacro*{in:}{}
\DeclareFieldFormat
  [article,inbook,incollection,inproceedings,patent,thesis,unpublished,misc]
  {title}{#1}

\setitemize{itemsep=0.02\baselineskip}

\newenvironment{enumeratearabic*}{
\begin{enumerate*}[label=(\arabic*)] %
}{
\end{enumerate*}
}

\newenvironment{enumerateroman*}{
\begin{enumerate*}[label=(\roman*)] %
}{
\end{enumerate*}
}

\numberwithin{equation}{section}

\newtheorem{theoremcounter}{theoremcounter}[section]

\theoremstyle{plain}

\newtheorem{lemma}[theoremcounter]{Lemma}

\newtheorem{theorem}[theoremcounter]{Theorem}

\theoremstyle{plain}

\theoremstyle{definition}

\theoremstyle{remark}

\newtheorem*{mainremarks}{Remarks}
\newtheorem*{mainexample}{Example}

\newcommand{\tx}{\text}
\makeatletter
\newcommand{\writelabel}[1]{#1\def\@currentlabel{#1}}
\makeatother

\newcommand{\minwidthmathbox}[2]{%
  \mathmakebox[{\ifdim#1<\width\width\else#1\fi}]{#2}%
}

\newcommand{\rmM}{\ensuremath{\mathrm{M}}}

\newcommand{\defeq}{\mathrel{:=}}

\newcommand{\mrelspace}[1]{\mathrel{\mspace{#1}}}

\NewCommandCopy{\rightarroworig}{\rightarrow}
\renewcommand{\rightarrow}
  {\protect\relbar\mrelspace{-9.7mu}\rightarroworig}

\NewCommandCopy{\leftarroworig}{\leftarrow}
\renewcommand{\leftarrow}
  {\protect\leftarroworig\mrelspace{-9.7mu}\relbar}

\renewcommand{\pmod}[1]{\;(\mathrm{mod}\, #1)}

\newcommand{\SL}[1]{\ensuremath{\mathrm{SL}_{#1}}}

\newcommand{\U}[1]{\ensuremath{\mathrm{U}_{#1}}}

\def\mt{\tilde m}
\def\jt{\tilde j}
\def\zp{\Z_{(p)}}
\def\SL{{\rm SL}}
\def\Z{\mathbb{Z}}
\def\F{\mathbb{F}}
\def\Q{\mathbb{Q}}
\renewcommand\tilde{\widetilde}

\title{Eisenstein series modulo $p^2$}

\author{Scott Ahlgren}
\address{Department of Mathematics\\
University of Illinois\\
Urbana, IL 61801} 
\email{sahlgren@illinois.edu} 

\author{Michael Hanson}
\address{Department of Mathematics\\
University of North Texas\\
Denton, TX 76203, USA}
\email{michael.hanson@unt.edu}

\author{Martin Raum}
\address{%
Chalmers tekniska högskola och G\"oteborgs Universitet,
Institutionen f\"or Matematiska vetenskaper,
SE-412 96 G\"oteborg, Sweden}
\email{martin@raum-brothers.eu}

\author{Olav K. Richter}
\address{Department of Mathematics\\
University of North Texas\\
Denton, TX 76203, USA}
\email{richter@unt.edu}

\thanks{The first author was partially supported by a grant from the Simons Foundation (\#963004 to Scott Ahlgren). The third author was partially supported by Vetenskapsr\aa det Grant~2023-04217.
The fourth author was partially supported by a grant from the Simons Foundation (\#835652 to Olav K. Richter).
}

\ifdraft{\date{\today\ at\ \thistime}}{\date{}}
\mathtoolsset{showonlyrefs}
\begin{document}

\begin{abstract}
We study congruences for Eisenstein series on $\SL_2(\Z)$ modulo $p^2$, where   $p\geq 5$ is prime.
It is classically known that all   Eisenstein series of weight at least $4$
are determined modulo~$p^2$  by those of weight at most
$p^2-p+2$.  We prove that up to  powers of $E_{p-1}$, each such  Eisenstein series  is  in fact determined modulo $p^2$ by a modular form of weight  at most~$2p-4$.  We also determine $E_2$ modulo $p^2$ in terms of a modular form of weight $p+1$.
\end{abstract}

\maketitle

\setcounter{tocdepth}{2}

\section{Introduction}
\label{sec:introduction}

Let $k\geq 2$ be even and let $B_k$ be the  Bernoulli number.  The Eisenstein series $G_k$ and $E_k$ defined by 
\begin{gather}
\label{eq:def:eisenstein_series}
  G_k
\defeq
  -\frac{B_k}{2k}\,E_k
\defeq
  -\frac{B_k}{2k} + \sum_{n=1}^\infty \sigma_{k-1}(n)\,q^n
\end{gather}
are basic objects  in the theory of modular forms.
It is well known \cite[\S 1]{Serre-p-adic}, \cite[\S3]{SwD-l-adic} that   if $p\geq 5$ is prime, then
\begin{alignat}{3}
\label{eq:modpcong}
    G_k &\equiv G_{k'} \quad&&\pmod p \qquad &\text{if } k &\equiv k' \not\equiv 0 \;\pmod {p-1} \tx{,}
\intertext{that}
\label{eq:ekone}
    E_k &\equiv 1 \quad&&\pmod p \qquad &\text{if } k &\equiv 0 \;\pmod{p-1} \tx{,}
\end{alignat}
and that 
\begin{gather}\label{eq:parEpm1}
    \partial E_{p-1}\equiv E_{p+1}\equiv E_2\quad\pmod p
\end{gather}
(where $\partial$ is the Ramanujan-Serre derivative defined in \eqref{eq:pardef}).

These facts
play a crucial role in the theory of modular forms modulo $p$.  
  Developed  by Swinnerton-Dyer \cite{SwD-l-adic}, this theory and its geometric extension due to Katz~\cite{Katz} has found many applications in the study of modular forms and their Galois representations;   it is  fundamental, for instance, 
to Serre's Conjecture (now a theorem of Khare--Wintenberger~\cite{Khare_Wintenberger_I,Khare_Wintenberger_II} and Kisin~\cite{Kisin}) and to the theory of Eisenstein ideals pioneered by Mazur \cite{Mazur}.

Serre~\cite[\S 1, 2]{Serre-p-adic} investigated $p$-adic properties of Eisenstein series.  In particular, he showed that $G_2=-E_2/24$ is a $p$-adic modular form; in other words, there is a sequence of modular forms  on $\SL_2(\Z)$ whose coefficients converge $p$-adically to those of $G_2$.  More recently, Chen and Kiming \cite[Cor.~2]{chen_kiming} determined an explicit $p$-adic expansion for $G_2$; in particular, they proved 
\cite[Prop.~2]{chen_kiming} that 
\begin{gather}
\label{eq:g2_chen_kiming}
  G_2
\equiv
  G_{2+p(p-1)} + p G_{p+1}^p
  \quad
  \pmod{p^2}
\tx{.}
\end{gather}

The goal of this paper is to obtain (in analogy with \eqref{eq:modpcong}--\eqref{eq:parEpm1} for the modulus $p$) a  description of all Eisenstein series modulo $p^2$ in terms of modular forms of low weight. 
 Before going further we introduce some notation and review what is classically known.   For $k$ even let $\rmM_k$ denote the space of modular forms of weight $k$ on $\SL_{2}(\Z)$ whose Fourier coefficients are rational.  We  identify each $f \in \rmM_k$ with its  Fourier expansion $f = \sum a(n)\, q^n\in \Q[\![q]\!]$.  Let $\zp$ be the ring of $p$-integral rational numbers.  By a $p$-integral modular form we mean a form whose coefficients lie in $\zp$, and by the statement  
$\sum a(n)\, q^n\equiv \sum  b(n)\, q^n \pmod {p^m}$ we mean that $a(n)\equiv b(n) \pmod {p^m}$ for all $n$.

Recall that $G_k \in \rmM_k$ if and only if $k\geq 4$.  Standard properties of Bernoulli numbers (which are recalled in the next section) imply that 
 \begin{gather}\label{eq:gkint}
     G_k \text{ is $p$-integral} \quad \tx{if and only if $(p-1) \nmid k$.}
 \end{gather}
In analogy with~\eqref{eq:modpcong} these properties also imply that if $k, k' \ge 4$,    then 
\begin{gather}\label{eq:Gp2}
  G_{k} \equiv G_{k'} \;\pmod{p^2} \ \ \ \ 
  \text{if $(p-1) \nmid k$\ \ \  and\ \ \  $k\equiv k' \;\pmod{p(p-1)}$}
\tx{.}
\end{gather}
 On the other hand, suppose that 
there is a congruence $f_1 \equiv f_2 \pmod{p^2}$ between 
$p$-integral modular forms $f_1$ and $f_2$ of weights $k_1$
and $k_2$, respectively, where
$f_1, f_2 \not\equiv 0 \pmod{p}$.  Then it is known~\cite[Thm.~1]{Serre-p-adic} that we have 
$k_1\equiv k_2 \pmod{p(p-1)}$.  
Combining these facts, we see that  each $G_k$ with $k\geq 4$ and $(p-1)\nmid k$
is congruent to exactly one Eisenstein series $G_{k_0}$ with $4\leq k_0\leq  2+p(p-1)$.
These facts are natural in view of the  fact that 
\begin{gather}\label{eq:eisp2}
    E_{p(p-1)}\equiv E_{p-1}^p\equiv 1 \quad\pmod{p^2} \tx{.}
\end{gather}
Our main result provides a finer description.    We show that, up to a power of $E_{p-1}$, every such $G_k$ is in fact determined by a modular form whose  weight is in the interval $[4,  2p-4]$.

\begin{theorem}
\label{thm:main:eisenstein}
Let $p \ge 5$ be prime, let~$2 \le k_0 \le p-3$ be even and let~$n \ge 0$. Then there exists a $p$-integral modular form $f_{k_0+(p-1)} \in \rmM_{k_0+(p-1)}$ such that 
\begin{gather}
  G_{k_0+(n+1)(p-1)}
\equiv
  E_{p-1}^n\, f_{k_0+(p-1)}
  \quad
  \pmod{p^2}
\tx{.}
\end{gather}
\end{theorem}

\begin{mainremarks}
\begin{enumerate}
\item By \eqref{eq:modpcong} and \eqref{eq:ekone}, we 
always have $f_{k_0+(p-1)}\equiv G_{k_0} \pmod p$.
\item
Recall that  $(p, k_0)$ is an irregular pair if and only if $p$ divides the numerator of $B_{k_0}$ (otherwise we will call it a regular pair).
If~$(p, k_0)$ is a regular pair, then by \eqref{eq:kummercong} the result also holds with $G_{k_0+(n+1)(p-1)}$ replaced by~$E_{k_0+(n+1)(p-1)}$.
\end{enumerate}
\end{mainremarks}

\begin{mainexample}
We give two examples to illustrate Theorem~\ref{thm:main:eisenstein}.
 We consider~$p=37$ and~$n=6$ in Theorem~\ref{thm:main:eisenstein}.
 Let $\Delta$ be the normalized cusp form of weight~$12$.
For the irregular pair $(37, 32)$ a computation proves the following congruence (where the modular form in parenthesis has weight $68=32+36$):
\begin{align*}
  G_{32 + 7 \cdot 36}
&\equiv 
  q + 652q^2 + 68 q^3 + \cdots
\\
&\equiv
  E_{36}^6 \cdot \big(
  E_4^{14} \Delta
  +
  498\, E_4^{11} \Delta^2
  + 
  1310\, E_4^8 \Delta^3
  +
  240\, E_4^5 \Delta^4
  +
  415\, E_4^2 \Delta^5
  \big)
  \quad
  \pmod{37^2}
\tx{.}
\end{align*}
For the regular pair $(37, 2)$  we have 
\begin{align*}
  E_{2 + 7\cdot 36}
&\equiv 
  1 + 272q + 705 q^2 + \cdots
\\
&\equiv
  E_{36}^6 \cdot \big(
  E_4^8 E_6
  +
  669\, E_4^5 E_6 \Delta
  +
  162\, E_4^2 E_6 \Delta^2
  \big)
  \quad
  \pmod{37^2}
\tx{,}
\end{align*}
where~$E_{2+7\cdot36}$ appears on the left side as explained in the second remark above.
\end{mainexample}

Theorem~\ref{thm:main:eisenstein} treats all weights apart from $2$ and  those which are divisible by $p-1$.  
To provide a complete picture, we treat these exceptions in the next two results (which by comparison 
have  simpler proofs).
  Recall  that  \eqref{eq:g2_chen_kiming} provides a description of $E_2$ modulo~$p^2$.
  The next result shows that $E_2\pmod {p^2}$ is in fact determined by a modular form of weight~$p+1$.
\begin{theorem}
\label{thm:main:e2}
Let $p\geq 5$ be prime. Then there exists a $p$-integral modular form $f_{p+1} \in \rmM_{p+1}$ such that 
\begin{alignat*}{2}
  E_2 E_{p-1}
&\equiv
  f_{p+1} + p E_{p+1}^p
  \quad
  &&
  \pmod{p^2}
\tx{,}
\intertext{or equivalently}
E_2&\equiv
   E_{p-1}^{p-1}\, f_{p+1} + p E_{p+1}^{p}
  \quad
  &&
  \pmod{p^2}
\tx{.}
\end{alignat*}
\end{theorem}

\begin{mainremarks}
\begin{enumerate}
\item The proof shows that we may take
$f_{p+1} =\partial E_{p-1}$ (compare with  \eqref{eq:parEpm1}).
\item By \eqref{eq:modpcong}, one can replace $E_2$ and $E_{p+1}$ by $G_2$ and $G_{p+1}$, respectively, in the statement.
\item 
The conclusion is equivalent to the congruence 
\begin{gather}
    E_2\equiv E_{p-1}^{2p-1} f_{p+1} + p E_{p-1}^{p-2}E_{p+1}^{p} \quad\pmod{p^2} \tx{,}
\end{gather}
where the right side is a modular form of weight~$2+2 p (p-1)$.
\item 
 Theorem~\ref{thm:main:e2} is a  corollary of \eqref{eq:g2_chen_kiming} and Theorem~\ref{thm:main:eisenstein} with $k_0 = 2$ and $n = p-1$. We will  give a   self-contained proof  in Section~\ref{sec:e2_congruences}.

\item
The Eisenstein series~$E_2$ appears 
naturally when differentiating modular forms (see \eqref{eq:thetadef}, \eqref{eq:pardef}).
For this reason, Theorem~\ref{thm:main:e2} 
is a tool to investigate the theta cycles of modular forms modulo~$p^2$;  this will be explored in a future project.
\end{enumerate}
\end{mainremarks}

\begin{mainexample}
As examples of Theorem~\ref{thm:main:e2} we have the following, which can easily be  proved by computation:
\begin{alignat*}{2}
  E_2E_{10} &\equiv -10\,E_4^3 + 15\,\Delta + 11\, E_{12}^{11} \quad&&\pmod{11^2}
\tx{,}
\\
  E_2E_{12} &\equiv -12\,E_{14} + 13\,E_{14}^{13} \quad&&\pmod{13^2}
\tx{,}
\\
  E_2E_{16} &\equiv -16\,E_4^3 E_6 + 219\,E_6 \Delta + 17\,E_{18}^{17} \quad&&\pmod{17^2}
\tx{.}
\end{alignat*}
\end{mainexample}

Finally we consider the case when the weight is divisible by $p-1$; here the result is particularly simple for any prime power modulus.
\begin{theorem}
\label{thm:main:epowers}
For any prime $p\geq 5$ and any positive integers $k$ and $r$ we have
\begin{gather}
  E_{rp^{k-1}(p-1)}
\equiv
  E_{p-1}^{rp^{k-1}}
  \quad
  \pmod{p^{k+1}}
\tx{.}
\end{gather}
\end{theorem}

In Section~\ref{sec:e2_congruences} 
we  prove
Theorems~\ref{thm:main:e2} and \ref{thm:main:epowers};
the arguments use   standard facts regarding modular forms and Bernoulli numbers.
The proof of Theorem~\ref{thm:main:eisenstein} in Section~\ref{sec:mainproof} is by contrast more difficult; here congruences involving Bernoulli numbers  and divisor sums will not suffice, since we have no information about the modular form $f_{k_0+(p-1)}$ that appears in the statement.
To prove the result, we instead leverage relationships among Eisenstein series and their derivatives due to Popa~\cite{popa} which arise from 
the Eichler-Shimura relations between period polynomials of modular forms. These  appear
when representing periods of modular forms via duality using Petersson scalar products (this is an insight going back to Rankin~\cite[(7.1)]{rankin-1952} and highlighted in the work of Kohnen--Zagier~\cite[p.~215]{kohnen-zagier-1984}). 
  A careful analysis of the $p$-adic properties of these relations
 allows us to use an inductive argument to prove Theorem~\ref{thm:main:eisenstein}.

\section{Background}
\label{sec:prelims}
We  recall some standard facts on Bernoulli numbers (see for example \S $9.5$ of \cite{Cohen}).  Let $p\geq 5$ be prime,  let $k, k',r$ be positive integers with $k, k'$ even, and let $v_p$ denote the $p$-adic valuation.
We have 
\begin{gather}\label{eq:adams} 
v_p\Big( \mfrac{B_k}k \Big)\geq 0\ \ \ \ \text{if $(p-1)\nmid k$}\text{,}
\end{gather}
and the Kummer congruences assert that if $(p-1)\nmid k$ and $k\equiv k' \pmod {p^{r-1}(p-1)}$, then 
\begin{gather}\label{eq:kummercong}
(1-p^{k-1})\mfrac{B_k}k\equiv
(1-p^{k'-1})\mfrac{B_{k'}}{k'} \quad\pmod
{p^r}.
\end{gather}
The Clausen-von Staudt theorem states that
\begin{gather}\label{eq:clausen}
B_k\equiv -\sum_{\substack{q \ \text{prime}\\ (q-1)\mid k}} \mfrac{1}{q} \quad\pmod{1}\text{,}
\end{gather}
from which it follows that 
\begin{gather} \label{eq:bkval}
v_p\Big( \mfrac{B_k}k \Big) =
-v_p(k)-1 \ \ \   \text{if $(p-1)\mid k$.}
\end{gather}
Moreover, if $(p - 1) \mid k$, then by \eqref{eq:clausen} we have 
\begin{gather}
pB_{k} \equiv-1 \quad\pmod{p}\text{,}
\end{gather}
from which it follows that 
\begin{gather}\label{eq:bkmodp2}
\mfrac{1}{B_{k}} \equiv -p \;\pmod{p^2} \ \ \ \ \text{if $(p-1)\mid k$.}
\end{gather}

We also recall the actions of the theta operator and the Ramanujan-Serre derivative on modular forms. If $f=\sum a(n)\,q^n\in \rmM_k$, then
\begin{gather}\label{eq:thetadef}
\theta f:=\sum na(n) \,q^n
\end{gather}
and
\begin{gather}\label{eq:pardef}
\partial f:=12\theta f-k E_2 f=12\theta f+24 k G_2 f\in \rmM_{k+2}.
\end{gather}
If $f\in \rmM_k$ is $p$-integral then it is well-known that there exists a $p$-integral $g\in \rmM_{k+p+1}$ such that $\theta f\equiv g \pmod p$.

\section{\texorpdfstring{Proof of Theorems~\ref{thm:main:e2} and \ref{thm:main:epowers}}{E2}}
\label{sec:e2_congruences}

These two results have relatively straightforward proofs
(as mentioned above, Theorem~\ref{thm:main:e2} is in fact a corollary of Theorem~\ref{thm:main:eisenstein} and \eqref{eq:g2_chen_kiming}). 
We treat them first since the arguments will be useful in the proof of the main theorem.
 
\begin{lemma}
 \label{lem:eptheta} For all primes $p\geq 5$ we have 
\begin{gather}\theta E_{p-1}\equiv
\mfrac{p}{12} \big(E_{p+1}-E_{p+1}^p \big) \quad\pmod{p^2}.
\end{gather} 
\end{lemma}
\begin{proof}
Using \eqref{eq:bkmodp2} and the fact that 
$n \sigma_{p-2}(n) \equiv \sigma_{p}(n)  \pmod{p}$ when $p\nmid n$, we find that 
\begin{gather}\label{eq:hasse_init}
\theta E_{p-1}= -\frac{2(p-1)}{B_{p-1}}\sum_{n=1}^{\infty}n\sigma_{p-2}(n)q^n \equiv -2p\sum_{\substack{n=1\\p\,\nmid\, n
}}^{\infty}\sigma_{p}(n)q^n \quad\pmod{p^2}.
\end{gather}
By Fermat's little theorem, the fact that $\sigma_p(pn)\equiv \sigma_p(n) \pmod p$, and \eqref{eq:kummercong} 
we have
\begin{gather}
  \sum_{\substack{n=1\\p\,\nmid \, n }}^{\infty}\sigma_{p}(n)\, q^n
\equiv
  -\mfrac{B_2}4 \big(E_{p+1}-E_{p+1}^p \big) \quad\pmod p. 
\end{gather}
The lemma follows from the two displayed equations.
\end{proof}

\begin{proof}
[Proof of Theorem~\ref{thm:main:e2}]
Define
\begin{align*}
f_{p+1}:=\partial E_{p-1} = 12\theta E_{p-1}  - (p-1)E_{2}E_{p-1}\in \rmM_{p+1}.
\end{align*}
By Lemma~\ref{lem:eptheta} we have
\begin{align*}
f_{p+1} \equiv p\big(E_{p+1} - E_{p+1}^{p}\big) - (p-1)E_{2}E_{p-1} \quad\pmod{p^2}.
\end{align*}
We conclude from the last equation that 
\begin{gather}
    E_{2}E_{p-1} \equiv f_{p+1} + pE_{p+1}^{p} \quad\pmod{p^2}\text{,}
\end{gather}
and the theorem follows from observing that~$E_2 E_{p-1} \equiv E_{p+1} \;\pmod{p}$.
\end{proof}

\begin{proof}[Proof of Theorem \ref{thm:main:epowers}]
By \eqref{eq:bkmodp2} we have  
\begin{gather}\label{eq:bernfact}
\frac{rp^{k-1}(p - 1)}{B_{rp^{k-1}(p-1)}} \equiv rp^{k} \quad\pmod{p^{k+1}}
\end{gather}
for any positive integers $k$ and $r$,
from which we see that
\begin{gather}\label{eq:lhs}
\begin{aligned}
E_{rp^{k-1}(p - 1)} &\equiv 1 - 2rp^k\sum_{n=1}^{\infty} \sigma_{rp^{k-1}(p - 1)-1}(n)q^n\\
&\equiv 
 1 - 2rp^k\sum_{n=1}^{\infty} \sigma_{p-2}(n)q^n \quad\pmod{p^{k+1}}.
\end{aligned}
\end{gather}
On the other hand, we have
\begin{gather}\label{eq:sum_expansion}
  E_{p-1}^{rp^{k-1}}
=
  1 +
  \sum_{j=1}^{rp^{k-1}}
  \mbinom{rp^{k-1}}{j}
  \Big( -\mfrac{2(p-1)}{B_{p-1}} \Big)^j\,
  \Big( \sum_{r=1}^{\infty} \sigma_{p-2}(n)q^n \Big)^j
\tx{.}
\end{gather}
From \eqref{eq:bkmodp2} we have $-2(p-1)/B_{p-1}\equiv -2p \pmod{p^2}$.
If   $2\leq j\leq rp^{k-1}$, then a straightforward calculation  shows that
\begin{gather}
v_p\Big(\mbinom{rp^{k-1}}{j}\Big) =v_p(rp^{k-1})-v_p(j)\geq k-1-v_p(j)\geq k+1-j.
\end{gather}
It follows  that
\begin{gather}
E_{p-1}^{rp^{k-1}}
\equiv 1 - 2rp^k\sum_{n=1}^{\infty}\sigma_{p-2}(n)q^n \quad\pmod{p^{k+1}},
\end{gather}
which together with \eqref{eq:lhs} proves the theorem.
\end{proof}

\section{Proof of Theorem~\ref{thm:main:eisenstein}}
\label{sec:mainproof}

As described in the Introduction, the proof of Theorem~\ref{thm:main:eisenstein} depends on 
a dual version of the Eichler-Shimura relations which manifests itself as a family of 
relations  between  Eisenstein series and their derivatives \cite{popa}. We begin with some notation which will be used throughout the section.  
Let $p\geq 5$ be prime, let $k$ be a positive integer, and for   $1\leq m\leq 2k-3$ define 
\begin{gather}
      \tilde m:=2k-2-m
\end{gather}      
(the definition of $\mt$ depends on $k$,  whose value will be clear from context).
Define
\begin{gather}\label{eq:ckmdef}
c(k, m):=\mfrac k{(m+1)(\mt+1)}+(-1)^m\mfrac{m!\,\mt!}{2(2k-1)!},
\end{gather}
and for $1\leq j\leq 2k-3$ define
\begin{gather} H_j:=\begin{cases}
0\ \ \ &\text{if $j$ is even,}\\
G_{j+1}G_{\jt+1}\ \ \ &\text{if $j$ is odd.}
\end{cases}
\end{gather}
Define
\begin{gather} 
F_m:=\sum_{\substack{j=3 \\ \text{odd}}}^m 
\mbinom{m}{j} H_j.
\end{gather}
Here we change slightly the notation in \cite{popa} to  allow for a  convenient restatement of the formulas (some extra care must be taken in the translation when $m=1, 2k-3)$.
Rewriting equation (A.3) of \cite{popa} with our notation gives the key formula
\begin{gather}\label{eq:popa}
c(k, m)\, G_{2k}=\mfrac1{24}\partial G_{2k-2}+F_m-H_m+F_{\mt},\ \ \ \ 1\leq m\leq 2k-3\text{.}
\end{gather}

Since the proof of the theorem is somewhat technical, we start with an overview.  For fixed $k_0$, we prove the congruence for $G_{k_0+(n+1)(p-1)}$ by induction on $n$.  The statement is trivially true when $n=0$, so by \eqref{eq:Gp2} it suffices to prove the result for $1\leq n\leq p-1$ when $k_0=2$.
When $k_0\geq 4$ the statement is also trivially true (by \eqref{eq:Gp2} and \eqref{eq:eisp2}) when $n=p-1$, so in this case we need only provide a proof for   $1\leq n\leq p-2$.   The proofs  for $n\leq p-2$ and for $n=p-1$  appear in the two subsections below and  require different choices of parameters in \eqref{eq:popa}.  
In each case we let $k=(k_0+(n+1)(p-1))/2$ (so that $G_{2k}$ is the series of interest), and
we  break the right side of \eqref{eq:popa} into three terms which we analyze separately:
\begin{gather}\label{eq:popa2}
  c(k, m)\,G_{2k}
=
  \underbrace{\mfrac1{24} \partial G_{2k-2}}_A
  +
  \underbrace{\vphantom{\mfrac1{24}} F_m-H_m}_B
  +
  \underbrace{\vphantom{\mfrac1{24}} F_{\mt}}_C
.
\end{gather}
The proof involves careful consideration of the $p$-adic valuations of the quantities which arise on each side.

We introduce the following  (slight abuse of) notation 
which leads to convenience in the proof: when we write 
\begin{gather}
f\equiv p^a\F_k \quad\pmod{p^b}
\end{gather}
we mean that there exists a $p$-integral form $g\in \rmM_k$ such that 
\begin{gather}
f\equiv p^a g \quad\pmod{p^b}.
\end{gather}
In the proof we will encounter many constants.  For ease of notation  we will use $\lambda_1, \lambda_2$, $\lambda_1'$  etc. to denote elements of $\zp$.

\subsection{\texorpdfstring{Proof for $n\leq p-2$}{p1}}
Let $k_0$ be an even integer with $2\leq k_0\leq p-3$.
Assume that $1\leq n\leq p-2$  and  that the result holds with $n$ replaced by $\alpha$
for $0\leq \alpha\leq n-1$.
For the parameters in \eqref{eq:popa2} we choose
\begin{gather}\label{eq:param}
 k=\mfrac{k_0+(n+1)(p-1)}2; \ \ \ \ m=p-1; \ \ \ \ \mt=k_0-2+n(p-1).
\end{gather}
We first examine the coefficient $c(k, m)$ in \eqref{eq:ckmdef}.
We find that 
\begin{gather}\label{eq:cmk1}
\frac k{(m+1)(\mt+1)}=
\begin{cases}
\frac{k_0}{2p(k_0-1)}\ \ \ &\text{if $n=k_0-1$} ,\\
\frac{\lambda_1}{p}\ \ \ &\text{if $n\neq k_0-1$,}
\end{cases}
\end{gather}
where $\lambda_1\in \zp$ has $\lambda_1\equiv 1/2 \pmod p$.
A straightforward computation shows  that 
\begin{gather}\label{eq:cmk2}
(-1)^m\frac{m!\,\mt!}{2(2k-1)!}=\frac{\lambda_2}p,
\end{gather}
where $\lambda_2\in \zp$ has  
\begin{gather}
\lambda_2\equiv 
\begin{cases}
\frac1{2n} \;\pmod p\ \ \ &\text{if $n\geq k_0-1$}, \\
\frac1{2(n+1)} \;\pmod p\ \ \ &\text{if $n< k_0-1$}. 
\end{cases}
\end{gather}
From this we conclude that in every case (recall that $n<p-1$ and that $k_0\leq p-3$) we have 
\begin{gather}\label{eq:ckmval}
c(k, m)=\mfrac{\lambda_3}p, \quad  \lambda_3\in \zp, \quad \lambda_3\not\equiv 0 \;\pmod p.
\end{gather}

We now analyze the three terms on  the right side of \eqref{eq:popa2}, 
which involve many products of the form $G_{j+1}G_{\jt+1}$.  
We record a lemma for convenience.
\begin{lemma}\label{lem:gprod}
Let $p\geq 5$ be prime, let  $k_0$ be an even integer with $2\leq k_0\leq p-3$,  let $n\geq0$ be an integer, and 
let $k$ be as in \eqref{eq:param}.  Suppose that  $j$ is an odd integer with $3\leq j\leq 2k-3$ and $j\not\equiv -1, k_0-1\pmod {p-1}$.
Then the following are true:
\begin{enumerate}
    \item  $G_{j+1}G_{\jt+1}$ is $p$-integral.
    \item With the additional assumption that  $j\not\equiv 1\pmod {p-1}$ when $k_0=4$, we have
    \begin{gather} G_{j+1}G_{\jt+1}\equiv \F_{k_0+(p-1)} \quad\pmod p.
    \end{gather}
\end{enumerate}
\end{lemma}
\begin{proof}  The first assertion is   clear from \eqref{eq:adams}.
For the second, 
define $j^*:=j\pmod{p-1}$, where $1 \le j^* \le p-4$ by assumption. 
Suppose that  $j^*\neq 1$.   
Then $G_{j^*+1}\in \rmM_{j^*+1}$,  and by \eqref{eq:modpcong} we have $G_{j+1}\equiv G_{j^*+1}\pmod p$ and
\begin{gather}
G_{\jt+1}=G_{k_0-1+(n+1)(p-1)-j}\equiv G_{p+k_0-2- j^*} \quad\pmod p.
\end{gather}
The result follows since the last Eisenstein series has weight $\geq 4$.

If  $j^*=1$ then the assumptions ensure that $k_0\neq 2,4$.
We have $G_{j+1}\equiv G_{p+1}\pmod p$, and
writing $j=1+\alpha(p-1)$ with $\alpha\geq 1$ we have 
\begin{gather}
G_{\jt+1}=G_{k_0-2+(n+1-\alpha)(p-1)}\equiv G_{k_0-2} \quad\pmod p.
\end{gather}
The desired congruence follows since $k_0-2\geq 4$.
\end{proof}

We now examine the second term in   \eqref{eq:popa2}.
We have 
\begin{gather}
B=\sum_{\substack{j=3\\  \text{odd}}}^{p-2} \mbinom{p-1}{j} G_{j+1}G_{k_0-1+(n+1)(p-1)-j}.
\end{gather}
By  Lemma~\ref{lem:gprod} there are at most two terms in the sum which are not $p$-integral.  These occur when  
\begin{gather}
j\equiv -1 \;\pmod{p-1} \quad\text{or}\quad k_0-1-j\equiv 0 \;\pmod{p-1}.
\end{gather}
  In other words, they 
arise from  $j=p-2$ and $j=k_0-1$ (the latter does not occur when $k_0=2$).
We conclude that there are $\lambda_4, \lambda_5\in \zp$  such that
\begin{gather}\label{eq:B}
  B
=
  \lambda_4\, G_{p-1}G_{k_0+n(p-1)}
  +
  \delta_{k_0\neq2}\cdot \lambda_5\, G_{k_0}G_{(n+1)(p-1)}
  +
  f_1
\tx{,}
\end{gather}
where  $f_1\equiv \F_{k_0+(p-1)}\pmod p$ by Lemma~\ref{lem:gprod}. 

For the third term in  \eqref{eq:popa2} we have
\begin{gather}\label{eq:C}
C=\sum_{\substack{j=3\\  \text{odd}}}^{k_0-3+n(p-1)} \mbinom{k_0-2+n(p-1)}{j} G_{j+1}G_{k_0-1+(n+1)(p-1)-j}.
\end{gather}
We  isolate the terms that may fail to be $p$-integral.
The terms with $j\equiv -1\pmod{p-1}$ arise from $j=\alpha(p-1)-1$, $1\leq\alpha\leq n$.
These contribute
\begin{gather}\label{eq:mt1}
  \sum_{\alpha=1}^n
  \lambda_\alpha'\, G_{\alpha(p-1)}G_{k_0+(n+1-\alpha)(p-1)}
\tx{,}\quad
  \lambda_\alpha'\in \zp
\tx{.}
\end{gather}
The terms with $j\equiv k_0-1\pmod{p-1}$ arise from $j=k_0-1+\beta(p-1)$
for  $0\leq\beta\leq n-1$  (if $k_0=2$, then there is no $\beta=0$ term).
These contribute
\begin{gather}\label{eq:mt2}
  \sum_{\beta=0}^{n-1}
  \lambda_\beta''\, G_{k_0+\beta(p-1)}G_{(n+1-\beta)(p-1)}
\tx{,}\quad
  \lambda_\beta''\in \zp
\tx{,}
  \ \ \lambda_0''=0 \ \ \text{if $k_0=2$}.
\end{gather}

We now consider several cases due to subtleties which arise when $k_0=2,4$. 
Suppose first that $k_0\neq 4$.  In this case \eqref{eq:B}, Lemma~\ref{lem:gprod},
   \eqref{eq:mt1} and \eqref{eq:mt2} yield
\begin{gather}\label{eq:mt3}
B+C=\sum_{\alpha=0}^{n} \lambda_\alpha''' G_{k_0+\alpha(p-1)}G_{(n+1-\alpha)(p-1)}+f_2,
 \ \ \ \lambda_\alpha'''\in \zp, \ \ \lambda_0'''=0\ \ \text{if $k_0=2$},
\end{gather}
where
$f_2\in \rmM_{k_0+(n+1)(p-1)}$ has 
$f_2\equiv \F_{k_0+(p-1)}\pmod p$. 
Combining \eqref{eq:popa2}, \eqref{eq:ckmval},  and \eqref{eq:mt3} and multiplying through by $p$ we find that
\begin{gather}\label{eq:multp}
\lambda_3 G_{k_0+(n+1)(p-1)}=p A+p \sum_{\alpha=0}^{n} \lambda_\alpha'''\, G_{k_0+\alpha(p-1)}G_{(n+1-\alpha)(p-1)}+p f_2, 
\end{gather}
where $\lambda_0'''=0$ if $k_0=2$.
Using \eqref{eq:bkval}, the induction hypothesis and Theorem~\ref{thm:main:epowers}  
we see that for each $\alpha$ in the sum there exists $\lambda_\alpha^*\in \zp$ with
\begin{gather}\label{eq:indterms}
\begin{aligned}
  p G_{k_0+\alpha(p-1)}G_{(n+1-\alpha)(p-1)}
&=
  \lambda_\alpha^*\, G_{k_0+\alpha(p-1)}E_{(n+1-\alpha)(p-1)}
\\
&\equiv
  E_{p-1}^{\alpha-1}E_{p-1}^{n+1-\alpha}\, \F_{k_0+(p-1)}
\equiv
  E_{p-1}^n\F_{k_0+(p-1)} \quad\pmod{p^2}.
\end{aligned}
\end{gather}
If, in addition, we have $k_0\neq 2$, then 
\begin{gather}\label{eq:An24}
A\equiv\mfrac1{24}\partial G_{k_0-2}\equiv \F_{k_0+(p-1)} \quad\pmod p.
\end{gather}
Combining  \eqref{eq:multp}, \eqref{eq:indterms}, and \eqref{eq:An24} completes the 
induction for  $n\leq p-2$  when $k_0\neq 2,4$.

When $k_0=2$, we use \eqref{eq:bkval} and  Theorem~\ref{thm:main:epowers} to find that there exists $\lambda_6\in \zp$ such that
\begin{gather}
  pA
=
  \lambda_6\, \partial E_{(n+1)(p-1)}
\equiv
  \lambda_6\, \partial E_{p-1}^{n+1}
\equiv
  \lambda_6\, (n+1) E_{p-1}^n\partial E_{p-1}
\equiv
  \F_{p+1} \quad\pmod p,
\end{gather}
which completes the induction for $n\leq p-2$ in this case.
 
Finally, suppose that  $k_0=4$.
In this case there are additional terms in \eqref{eq:C} to which Lemma~\ref{lem:gprod} does not apply.  These   arise from  the indices $j$ with $j\equiv 1\pmod{p-1}$
and 
 combine to give 
\begin{gather}\label{eq:4bonus}
\sum_{\alpha=1}^n \mbinom{2+n(p-1)}{1+\alpha(p-1)} G_{2+\alpha(p-1)}G_{2+(n+1-\alpha)(p-1)}\equiv n G_{p+1}^2 \quad\pmod p.
\end{gather}
Here the  binomial sum congruence can be justified using 
 Lucas' Theorem (recalling that $n\leq p-2$) to see that 
\begin{gather}
\sum_{\alpha=1}^n \mbinom{2+n(p-1)}{1+\alpha(p-1)} \equiv  \mbinom{n-1}{1}\mbinom{p-n+2}{0}+\sum_{\alpha = 2}^{n} \mbinom{n-1}{\alpha-1}\mbinom{p-n+2}{p-\alpha+1}  \quad\pmod p.
\end{gather}
Since the second binomial coefficient in each summand is zero for $2 \leq \alpha \leq n-2$, we have 
\begin{gather}
\sum_{\alpha=1}^n \mbinom{2+n(p-1)}{1+\alpha(p-1)} \equiv  (n-1)+(n-1)+(2-n)\equiv n  \quad\pmod p.
\end{gather}

The  additional term \eqref{eq:4bonus}
can be canceled by extracting the corresponding quantity
from the first term in \eqref{eq:popa2}:
\begin{lemma}\label{lem:A4}
 If  $p\geq 5$ is prime and $n\geq 0$, then 
\begin{gather}
A=\mfrac1{24}\partial G_{2+(n+1)(p-1)}\equiv -n G_{p+1}^2 +\mfrac1{24} \partial G_{p+1} \quad\pmod p.
\end{gather}
\end{lemma}
\begin{proof} 
Using~\eqref{eq:pardef} and reducing modulo $p$ we obtain
\begin{gather}
A \equiv \mfrac{1}{2}\theta G_{2 + (n+1)(p-1)} + (1-n)G_{2}G_{2 + (n+1)(p-1)}  \quad\pmod{p}.
\end{gather}
The result follows, since $G_{2 + (n+1)(p-1)} \equiv G_2\equiv G_{p+1} \pmod{p}$.
\end{proof}

Using \eqref{eq:4bonus}, Lemma~\ref{lem:A4}, and the fact that 
$\partial G_{p+1}\in \rmM_{4+(p-1)}$ we find that the analogue of 
\eqref{eq:multp} when $k_0=4$ is 
\begin{gather}\label{eq:multp4}
\lambda_3 G_{4+(n+1)(p-1)}=p \sum_{\alpha=0}^{n} \lambda_\alpha''' G_{4+\alpha(p-1)}G_{(n+1-\alpha)(p-1)}+p f_3, 
\end{gather}
where $f_3\equiv \F_{4+(p-1)}\pmod p$.
Reducing modulo $p^2$ and arguing inductively as in \eqref{eq:indterms} gives the result in this case.
This completes the induction for $n\leq p-2$  (and so finishes the proof for $k_0\neq 2$).

\subsection{\texorpdfstring{Proof when $n=p-1$, $k_0=2$}{p2}}
The proof in this case follows the same strategy with a different choice of parameters in \eqref{eq:popa2} (which is necessary to obtain the analogue of \eqref{eq:ckmval}).  Here we choose
\begin{gather}\label{eq:param1}
 k=\mfrac{2+p(p-1)}2; \ \ \ \ m=p-2; \ \ \ \ \mt=p^2-2p+2.
\end{gather}
A straightforward computation using \eqref{eq:ckmdef} shows that the left side of \eqref{eq:popa2}
becomes
\begin{gather}\label{eq:left2}
\mfrac{\lambda_1}p G_{2+p(p-1)},\ \ \ \  \lambda_1\in \zp, \ \ \ \  \lambda_1\equiv-1 \;\pmod p.
\end{gather}
Using \eqref{eq:bkval} we find that
\begin{gather}\label{eq:A2}
A=\mfrac1{24}\partial G_{p^2-p}=\mfrac{\lambda_2}{p^2}\, \partial E_{p^2-p},  \ \ \ \  \lambda_2\in \zp.
\end{gather}
In this case we have 
\begin{gather}
B=F_{p-2}-H_{p-2}=\sum_{\substack{j=3\\  \text{odd}}}^{p-4} \mbinom{p-2}{j} G_{j+1}G_{1+p(p-1)-j},
\end{gather}
and by Lemma~\ref{lem:gprod} we find that
\begin{gather}\label{eq:B2}
B \equiv \F_{2+(p-1)} \quad\pmod p.
\end{gather}
We  have
\begin{gather}
C=\sum_{\substack{j=3\\  \text{odd}}}^{(p-1)^2+1} \mbinom{(p-1)^2+1}{j} G_{j+1}G_{1+p(p-1)-j},
\end{gather}
and arguing as in \eqref{eq:mt3} we find that
\begin{gather}\label{eq:C2}
B+C=\sum_{\alpha=1}^{p-1} \lambda_\alpha'\, G_{2+\alpha(p-1)}G_{(p-\alpha)(p-1)}+g_1,
 \ \ \ \lambda_\alpha'\in \zp,
\end{gather}
where 
$g_1\equiv \F_{2+(p-1)}\pmod p$.
Combining \eqref{eq:left2}, \eqref{eq:A2}, \eqref{eq:B2} and \eqref{eq:C2} and multiplying through by $p$ gives
\begin{gather}\label{eq:G22}
\lambda_1\, G_{2+p(p-1)}=\mfrac{\lambda_2}{p}\,\partial E_{p^2-p}+p \sum_{\alpha=1}^{p-1} \lambda_\alpha'\, G_{2+\alpha(p-1)}G_{(p-\alpha)(p-1)}+p g_2,
\end{gather}
where 
$g_2\equiv \F_{2+(p-1)}\pmod p$.
To determine the first term on the right modulo $p^2$ we use the following result.
\begin{lemma} \label{lem:partialE} 
If  $p\geq 5$ is prime then
\begin{gather} \partial E_{p^2-p}\equiv pE_{p-1}^{p-1}E_{p+1}+p^2\F_{p+1} \quad\pmod{p^3}.
\end{gather}
\end{lemma}
\begin{proof} By Theorem~\ref{thm:main:epowers} we have 
\begin{gather} \partial E_{p^2-p}\equiv \partial E_{p-1}^p\equiv p E_{p-1}^{p-1}\partial E_{p-1} \quad\pmod{p^3}.
\end{gather}
Here $\partial E_{p-1}\in \rmM_{p+1}$ has  $\partial E_{p-1}\equiv E_{p+1}\pmod p$; in 
other words,
\begin{gather}
\partial E_{p-1}\equiv E_{p+1}+p \F_{p+1} \quad\pmod{p^2}.  
\end{gather}
The Lemma follows.
\end{proof}
From Lemma~\ref{lem:partialE}  we have 
\begin{gather}\mfrac{\lambda_2}{p}\partial E_{p^2-p}\equiv 
E_{p-1}^{p-1}E_{p+1}+p\F_{p+1} \quad\pmod{p^2}.
\end{gather}
For each $\alpha$ appearing in \eqref{eq:G22}, the results of the last section and Theorem~\ref{thm:main:epowers}  show that for some $\lambda_\alpha^*\in \zp$ we have
\begin{gather}\label{eq:indterms2}
\begin{aligned}
  p G_{2+\alpha(p-1)}G_{(p-\alpha)(p-1)}
&=
  \lambda_\alpha^*\, G_{2+\alpha(p-1)}E_{(p-\alpha)(p-1)}
\\
&\equiv
  E_{p-1}^{\alpha-1}E_{p-1}^{p-\alpha}\F_{2+(p-1)}
\equiv
  E_{p-1}^{p-1}\F_{2+(p-1)} \quad\pmod{p^2}.
\end{aligned}
\end{gather}
Using the last two  facts, we see from \eqref{eq:G22} that
\begin{gather}
  \lambda_1\, G_{2+p(p-1)}
\equiv
  E_{p-1}^{p-1}\F_{2+(p-1)} \quad\pmod{p^2}.
\end{gather}
This establishes the result when $k_0=2$ and $n=p-1$ and so completes the proof of Theorem~\ref{thm:main:eisenstein}.

\printbibliography

\end{document}